\documentclass{birkjour}
 \newtheorem{thm}{Theorem}[section]
 \newtheorem{cor}[thm]{Corollary}
 \newtheorem{lem}[thm]{Lemma}
 
 \theoremstyle{definition}
 \newtheorem{defn}[thm]{Definition}

 \numberwithin{equation}{section}

\usepackage{amsfonts,amsmath,amssymb}

\usepackage[colorlinks]{hyperref}

\usepackage{enumerate}

\usepackage{tikz}

\usepackage{tikz-cd}

\numberwithin{equation}{section}

\newcommand{\cl}[1]{\mathcal{#1}} 

\newcommand{\bb}[1]{\mathbb{#1}}

\newcommand{\sca}[1]{\left\langle#1\right\rangle} 

\begin{document}

\title[A new approach to the similarity problem]
 {A new approach to the similarity problem}

\author{E. Papapetros}
\address{University of Patras\\Faculty of Sciences\\ Department of Mathematics\\265 00 Patras Greece}
\email{e.papapetros@upatras.gr}



\maketitle

\begin{abstract}
We say that a $C^*$-algebra $\cl A$ satisfies the similarity property ((SP)) if every bounded homomorphism $u\colon \cl A\to \cl B(H),$ where $H$ is a Hilbert space, is similar to a $*$-homomorphism. We introduce the following hypothesis {\bf (EP).}\\
  {\bf (EP): Every separably acting von Neumann algebra with a cyclic vector is hyperreflexive.}
\\We prove that under {\bf (EP)}, all $C^*$-algebras satisfy (SP).
\end{abstract}

\section{Introduction}
In 1955, Kadison in his article \cite{kad} introduced the following problem which still remains one of the most famous open problems in the theory of $C^*$-algebras: 

Let $H$ be a Hilbert space and $u\colon \cl A\to \cl B(H)$ be a bounded homomorphism on a $C^*$-algebra $\cl A.$ Does there exist an invertible operator $S\in\cl B(H)$ such that the map $$\pi(A)=S^{-1} u(A) S,\,\,A\in\cl A$$ defines a $*$-homomorphism?\\
We say 
that a $C^*$-algebra $\cl A$ satisfies the similarity property ((SP)) if every bounded homomorphism $u\colon \cl A\rightarrow \cl B(H),$ where $H$ is a Hilbert space, 
has the above property. In  \cite{haag}, Haagerup gave an affirmative answer for the case in which the map $u$ has a finite cyclic set, i.e. there exist $n\in\bb N$ and $\xi _1, ...,
\xi_n\in H$ such that $$H=\overline{[u(A)\xi _i\,\mid\,\,A\in\cl A,\,\,i=1,...,n]},$$ where, in general, for a subset $\cl S$ of some vector space $\cl X,$ $\,[\cl S]$ denotes the linear span of $\cl S.$  
In the same article, Haagerup proved that  every completely bounded homomorphism from a $C^*$-algebra $\cl A$ into $\cl B(H),$ where $H$ is a Hilbert space, is similar to a $*$-homomorphism. 

Kadison's similarity problem is equivalent to\\
 {\bf Arveson's hyperreflexivity problem : Every von Neumann algebra is hyperreflexive}.
 Arveson in \cite{arv} showed that all nest algebras are hyperreflexive with hyperreflexivity constant $1.$  Although many von Neumann algebras are hyperreflexive, the question of whether every von Neumann algebra is hyperreflexive remains open. 
 
We say that a von Neumann algebra $\cl M$ satisfies the weak similarity property ((WSP)) if  
every w*-continuous and unital bounded homomorphism $u\colon \cl M\to \cl B(H),$ where $H$ is a Hilbert space, is similar to a $*$-homomorphism. The connection between (SP) and (WSP) is given by the following lemma. Its proof can be found in \cite{ele-pap}.

\begin{lem}
\label{WSP}
Let $\cl A$ be a unital $C^*$-algebra. The algebra $\cl A$ satisfies (SP) if and only if the second dual algebra $\cl A^{**}$ satisfies (WSP).
\end{lem}

Recently, the author and Eleftherakis in their joint work \cite{ele-pap} proved the following theorem.

\begin{thm}
A von Neumann algebra $\cl M$ satisfies (WSP) if and only if the von Neumann algebras $\cl M^{\prime}\bar \otimes \cl B(\ell^2(I))$ are hyperreflexive for all cardinals $I.$ Here, $\cl M^{\prime}$ is the commutant of $\cl M.$  
\end{thm}
In the special case of a separably acting von Neumann algebra $\cl M,$ we proved the following corollary.

\begin{cor}
    Let $\cl M\subseteq \cl B(H)$ be a von Neumann algebra and $H$ be a separable Hilbert space. Then $\cl M$ satisfies (WSP) if and only if the algebra $\cl M^\prime \bar \otimes \cl B(\ell^2(\bb{N}))$ is hyperreflexive.
\end{cor}
In general, a $\mathrm{w}^*$-closed space $\cl X$ is called completely hyperreflexive if the space  $\cl X\bar\otimes \cl B(\ell^2(\bb N))$ is hyperreflexive. It is obvious that completely hyperreflexive spaces are hyperreflexive. On the other hand, whether hyperreflexivity implies complete hyperreflexivity remains an open problem \cite{do, dl}.

In the same article \cite{ele-pap}, the author and Eleftherakis introduced the following hypothesis:\\

{\bf (CHH):} Every hyperreflexive separably acting von Neumann algebra is completely hyperreflexive.\\ 

As we showed in Section 3 of our article, the resolution of this problem yields an affirmative answer to the similarity problem. 

Throughout this paper, if $\cl M$ and $\cl N$ are von Neumann algebras, then we denote by 
$\cl M\bar \otimes \cl N$ their spatial tensor product. If $I$ is a cardinal, then $M_{I}(\cl M)$ is the dual operator space of $I\times I$ matrices with entries in $\cl M$ which define bounded operators. If $I=\bb N$ we simply write $M_{\infty}(\cl M).$ See \cite{bm,pau,pisbook} for further details.

We introduce the following hypothesis:\\

{\bf (EP): Every separably acting von Neumann algebra with a cyclic vector is hyperreflexive.}\\
We underline that many separably acting von Neumann algebras with a cyclic vector are hyperreflexive. For example, every separably acting von Neumann algebra $\cl M$ which is in standard position has a cyclic vector, it admits also a separating vector and by \cite[Corollary 3.4]{kl} it follows that $\cl M$ is hyperreflexive.

The hypothesis we introduce is weaker than Arveson's. Arveson conjectured that all von Neumann algebras are hyperreflexive, while we restrict on the ones that are separably acting and have a cyclic vector.

In Section 2, we give some basic definitions and we set the stage by presenting useful theorems and lemmas that were proven in \cite{ele-pap} and are useful for the next section.

In Section 3, we prove that under {\bf (EP)}, every separably acting von Neumann algebra satisfies (WSP). Using this result, we show that under {\bf (EP)}, all $C^*$-algebras satisfy (SP).

In what follows, $\,\cl B(H)$ is the algebra of all linear and bounded operators from the Hilbert space $H$ to itself. Furthermore,  
if $\cl D$ is a set of operators acting on $H$ then $$\cl D^\prime=\left\{T\in\cl B(H)\mid T D=D T ,\,\forall\,D\in\cl D\right\}$$ is the commutant of $\cl D.$

\section{Preliminaries}
In this section we recall some basic definitions and facts that will be used throughout the article.

Let $\cl M$ be a von Neumann algebra acting on the Hilbert space $H.$ If $T\in \cl B(H)$ we set
$$d(T,\,\cl M) =\inf\left\{\|T-X\|\,\,\mid X\in\cl M\right\}$$ 
to be the distance from $T$ to $\cl M$. 
We also set
$$r_{\cl M}(T) =\sup_{\|\xi \|=\|\eta \|=1} \left\{\left|\sca{T\xi ,\eta }\right|\,\,\,\,\mid  \sca{X\xi, \eta }=0,\,\forall\,X\in \cl M\right\}$$ 
and it is easy to prove that $r_{\cl M}(T)\leq d(T,\,\cl M).$ Furthermore, it is immediate that the quantities $d(\cdot,\,\cl M)$ and $r_{\cl M}(\cdot)$ define seminorms on $\cl B(H)$ and the quotient space $\cl B(H)/\cl M$ is a Banach space with respect to the norm $$||T+\cl M||_{1}=d(T,\,\cl M),\,\,T\in\cl B(H).$$
\begin{lem}
\label{denominator}
If $\,T\notin \cl M,$ then $r_{\cl M}(T)\neq 0.$
\end{lem}
The above lemma states that the quantity $$||T+\cl M||_{2}=r_{\cl M}(T),\,\,T\in\cl B(H)$$ defines a norm on $\cl B(H)/\cl M$ but the quotient space $\cl B(H)/\cl M$ is not necessarilly a Banach space with respect to this norm.

\begin{defn}
\label{hyper}
If there exists $0<k<\infty$ such that 
$$d(T,\,\cl M) \leq k\,r_{\cl M}(T),\,\forall\,T\in \cl B(H),$$
then we say that the von Neumann algebra $\cl M$ is hyperreflexive.
\end{defn}
We observe that if $\cl M$ is hyperreflexive, then the quotient space $\cl B(H)/\cl M$ is a Banach space with respect to the norm $||\cdot||_{2}$ since in that case, the norms $||\cdot||_{1}$ and $||\cdot||_{2}$ are equivalent. 

Lemma \ref{denominator} and Definition \ref{hyper} lead to the following definition.

\begin{defn}
Let $\cl M$ be a von Neumann algebra acting on the Hilbert space $H.$ We define 
$$k(\cl M)=\sup_{T\not \in\,  \cl M}\frac{d(T,\, \cl M) }{r_{\cl M}(T) }$$ to be the hyperreflexivity constant of $\,\cl M.$
\end{defn}
Therefore, $\cl M$ is hyperreflexive if and only if $k(\cl M)<\infty.$ Moreover, if $\cl M$ is hyperreflexive, then $k(\cl M)\geq 1$ (since $0<r_{\cl M}(T)\leq d(T,\,\cl M)).$
 
In \cite{ele-pap}, the author and Eleftherakis, using arguments from Pisier's book \cite{pisier1}, proved some results connecting the weak similarity property with the notion of hyperreflexivity of von Neumann algebras. The most important of them are the following:

\begin{thm} \cite{ele-pap}
A von Neumann algebra $\cl M$ satisfies (WSP) if and only if the von Neumann algebras $\cl M^{\prime}\bar \otimes \cl B(\ell^2(I))$ are hyperreflexive for all cardinals $I.$ 
\end{thm}

\begin{cor} \cite{ele-pap}
\label{104}
    Let $\cl M\subseteq \cl B(H)$ be a von Neumann algebra and $H$ be a separable Hilbert space. Then $\cl M$ satisfies (WSP) if and only if the algebra $\cl M^\prime \bar \otimes \cl B(\ell^2(\bb{N}))$ is hyperreflexive.
\end{cor}

Concluding this section, we remind the reader the definition of a completely bounded map between von Neumann algebras.

\begin{defn}
 Let $\cl M$ and $\cl N$ be von Neumann algebras acting on the Hilbert spaces $H$ and $K$ respectively, and $u\colon \cl M\to \cl N$ be a linear map. For each $n\in\bb N$ consider the matrix amplification $$u_n\colon M_n(\cl M)\to M_n(\cl N),\,\, u_n((X_{i,j}))=(u(X_{i,j})).$$
We say that $u$ is completely bounded if $$||u||_{cb}:=\sup_{n} ||u_n||<\infty.$$
\end{defn}

\section{(EP) implies that all C*-algebras satisfy (SP)}

In this section we prove that under {\bf (EP)}, every separably acting von Neumann algebra satisfies (WSP). As we will see, Corollary \ref{104} plays a crucial role for the proof of this result. Using this result, we show that under {\bf (EP)}, all $C^*$-algebras satisfy (SP). First, we prove some useful lemmas.

\begin{lem}
\label{sv}
Assume that {\bf (EP)} is true. Then every separably acting von Neumann with a separating vector satisfies (WSP).
\end{lem}

\begin{proof}
 Let $\cl M$ be a separably acting von Neumann algebra with a separating vector. Then the separably acting von Neumann algebra $\cl M^\prime\bar \otimes \cl B(\ell^2(\bb N))$ has a cyclic vector and under our hypothesis we have $k(\cl M^\prime\bar \otimes \cl B(\ell^2(\bb N)))<\infty.$ According to Corollary \ref{104} we deduce that the von Neumann algebra $\cl M$ satisfies (WSP).
\end{proof}


\begin{lem}
\label{SOL1}
    Assume that {\bf (EP)} is true. Then every separably acting von Neumann algebra satisfies (WSP).
\end{lem}

\begin{proof}
  Let $\cl M$ be a von Neumann algebra acting on the separable Hilbert space $H.$ By \cite[Theorem 1.6]{haag2} the algebra $\cl M$ is isomorphic to a von Neumann algebra $\cl N$ which is standard on some Hilbert space $K.$ Since $H$ is separable, the von Neumann algebra $\cl M$ is $\sigma$-finite. Therefore, $\cl N$ is $\sigma$-finite and by Lemma 2.8 in \cite{haag2} the algebra $\cl N$ admits a cyclic and a separating vector. Lemma \ref{sv} yields that $\cl N$ satisfies (WSP), and thus $\cl M$ satisfies (WSP).
\end{proof}
Lemmas \ref{extend} and \ref{RR} were proven in \cite{ele-pap}. We present them here, as well, for the sake of completeness.
\begin{lem} \cite{ele-pap}
\label{extend}
    Let $\cl X$ be a dual Banach space, let $\cl X_0,\,\cl Y,\,\cl Z$ be Banach spaces such that 
$\cl X=\overline{\cl X_0}^{\mathrm{w}^*}$ and $\phi\colon \cl X\to \cl Y^{**}$ be a $\mathrm{w}^*$-continuous onto isometry such that $\phi(\cl X_0)=\cl Y.$ If $u\colon \cl X_0\to \cl Z^*$ is a bounded linear map, then there exists a $\mathrm{w}^*$-continuous bounded linear map $\tilde u\colon \cl X\to \cl Z^*$ such that $\tilde{u}|_{\cl X_0}=u$ and $\|u\|=\|\tilde u\|.$
\end{lem}

\begin{lem}\cite{ele-pap}
\label{RR} Let $\mathcal A$ be a $\,C^*$-algebra such that $\mathcal A\subseteq \mathcal A^{**}\subseteq \mathcal B(H)$ and let $P\in \mathcal A^\prime$ be a projection. Then there exists a $C^*$-algebra $\mathcal D$ 
and a $*$-isomorphism  $\alpha\colon \mathcal A P\to \mathcal D$ which extends to a $*$-isomorphism from 
$\overline{\mathcal A P}^{\mathrm{w}^*}$ to $\mathcal D^{**}.$ 
\end{lem}

The above Lemmas are the key to pass from every separably acting von Neumann algebra to all $C^*$-algebras. Assuming that {\bf (EP)} holds and following exactly the same proof procedure for Lemma 3.5, Theorem 3.6 and Corollary 3.7 in \cite{ele-pap} we deduce that

\begin{thm}
    Under {\bf (EP)}, all $C^*$-algebras satisfy (SP).
\end{thm}


\end{document}